\documentclass[12pt]{article}
\usepackage{amsmath}
\usepackage{latexsym}
\usepackage{amssymb}
\newtheorem{thm}{Theorem}[section]
\newtheorem{la}[thm]{Lemma}
\newtheorem{Defn}[thm]{Definition}
\newtheorem{Remark}[thm]{Remark}
\newtheorem{Conj}[thm]{Conjecture}

\newtheorem{Example}[thm]{Example}
\newtheorem{Number}[thm]{\!\!}
\newenvironment{defn}{\begin{Defn}\rm}{\end{Defn}}

\newenvironment{rem}{\begin{Remark}\rm}{\end{Remark}}

\newenvironment{numba}{\begin{Number}\rm}{\end{Number}}
\newenvironment{proof}{{\noindent\bf Proof.}}%
                  {\nopagebreak\hspace*{\fill}$\Box$\medskip\par}
\newcommand{\Punkt}{\nopagebreak\hspace*{\fill}$\Box$}

\newcommand{\at}{\symbol{'100}}

\newcommand{\impl}{\Rightarrow}
\newcommand{\mto}{\mapsto}

\DeclareMathOperator{\Ad}{Ad}
\newcommand{\N}{{\mathbb N}}

\newcommand{\Q}{{\mathbb Q}}
\newcommand{\Z}{{\mathbb Z}}
\newcommand{\cg}{{\mathfrak g}}

\newcommand{\one}{{\bf 1}}
\DeclareMathOperator{\Aut}{Aut}
\newcommand{\sub}{\subseteq}
\DeclareMathOperator{\GL}{GL}
\DeclareMathOperator{\id}{id}

\DeclareMathOperator{\rad}{rad}
\begin{document}
%
%
%
\begin{center}
{\Large\bf The kernel of the adjoint\vspace{1mm}
representation of a {\boldmath$p$}-adic Lie group need not have
an abelian\\[2.5mm]
open normal subgroup}\\[7mm]
{\bf Helge Gl\"{o}ckner}\vspace{2mm}
\end{center}
\begin{abstract}
\hspace*{-7mm}
Let $G$ be a $p$-adic Lie group with Lie algebra
$\cg$ and
$\Ad\colon G\to\Aut(\cg)$ be the adjoint representation.
It was claimed in the literature
that the kernel $K:=\ker(\Ad)$
always has an abelian open normal subgroup.
We show by means of a counterexample
that this assertion is false;
it can even happen that $K=G$ but $G$ has no abelian subnormal
subgroup except for the trivial group.
The arguments are based on auxiliary results
on subgroups of free
products with central amalgamation.\vspace{2mm}
\end{abstract}
{\bf Classification:} Primary 22E20;
secondary 20D35,
20E06,
22E25,
22E50\\[3mm]
%
%
{\bf\Large Introduction and statement of results}\\[4mm]
The context of these investigations are recent studies
of normal series
\[
G=G_0\rhd G_1\rhd \cdots \rhd G_n=\one
\]
of closed subgroups in a $p$-adic Lie group $G$.
Although composition series need not exist
(as is evident from
$\Z_p\rhd p\Z_p\rhd p^2\Z_p\rhd\cdots$),
by reasons of dimension one can always
find a series all of whose subquotients
$Q:=G_{j-1}/G_j$ are \emph{nearly simple}
in the sense that each closed subnormal
subgroup $S\sub Q$ is open or discrete~\cite{ELE}.
It turns out that a quite limited list
of subquotients suffices to build up~$G$
(\emph{loc.\,cit.}).
The kernel of the adjoint representation
is an important source of normal
subgroups,
and its properties
have an impact on the list of subquotients
needed to build up general $p$-adic Lie groups
(see the proposition below).
Some statements from the literature
are relevant in this connection.
In \cite[Exercise 9.12]{Dix},
the following assertion is made:\\[3mm]
{\bf Assertion 1.}
\emph{Let $G$ be a $p$-adic Lie group.
Then $G$ has closed normal subgroups $Z\sub K$
such that $Z$ is abelian, $K/Z$ is discrete,
and $G/K$ is isomorphic to a subgroup of $\GL_d(\Q_p)$
where $d=\dim(G)$.}\\[2.5mm]
The hint given in the exercise amounts to the following assertion:\\[3mm]
{\bf Assertion 2.}
\emph{If $G$ is a $p$-adic Lie group,
with Lie algebra $\cg$ and adjoint representation
$\Ad\colon G\to \GL(\cg)$,
then the kernel $\ker(\Ad)$ has an abelian open
subgroup which is normal in $G$.}\\[3mm]
In fact, the hint to solve the exercise
is to take $K:=\ker(\Ad)$
and define $G^+$ as the subgroup
generated by all compact open subgroups
of $G$ which are uniformly powerful.
This is an open normal subgroup of $G$ (and so $G/G^+$
and $K/(K\cap G^+)$ are discrete).
The reader is supposed to show that $K$ is the centralizer of $G^+$
in~$G$ (which is false), which would entail
that $K$ and $Z:=K\cap G^+$ can be used to obtain Assertion~1.\\[3mm]
The goal of this note is to show that
Assertions~1 and~2 are false,
and to record additional pathologies
that may occur. However, a weakened version
of Assertion~1 cannot be ruled out so far:\\[3mm]
{\bf Problem 1.}
Does every $p$-adic Lie group $G$ have
closed subgroups $Z\lhd K\lhd G$ such that $Z$ is abelian, $K/Z$ is discrete
and $K=\ker\beta$ for a continuous homomorphism $\beta\colon G\to \GL_d(\Q_p)$
for some $d\in \N$?\\[3mm]
Our counterexample to Assertions 1 and 2
reads as follows:\\[3mm]
{\bf Theorem.}
\emph{The ascending union}
\[
G:=\bigcup_{n\in \N_0} (\Q_p*_{\Z_p} \Q_p *_{p\Z_p}\cdots
*_{p^n\Z_p}\Q_p)
\]
\emph{of iterated amalgamated products
can be made a $1$-dimensional
$p$-adic Lie group
with the following properties}:
\begin{itemize}
\item[(a)]
\emph{$\Ad(g)=\id$ for all $g\in G$, i.e., $\ker(\Ad)=G$};
\item[(b)]
\emph{$G$ does not have an abelian non-trivial subnormal subgroup};
\item[(c)]
\emph{$G$ has a discrete normal subgroup $K$
such that $G/K\cong \Q_p$, whence
$K=\ker(\psi)$ for some analytic homomorphism
$\psi\colon G\to \GL_2(\Q_p)$}; \emph{and}:
\item[(d)]
\emph{The commutator group $G'$ is discrete}.
\end{itemize}
{\bf Remarks.}
\begin{itemize}
\item[(i)]
By (a) and (b), $\ker(\Ad)=G$ does not have an
abelian open normal subgroup, showing that
Assertion~2 is false.
\item[(ii)]
Our group $G$ also shows that Assertion~1 is false.\footnote{In fact, $d=\dim(G)=1$ here. Suppose that $K:=\ker\beta$
had an abelian open normal subgroup $Z$ for some
homomorphism $\beta\colon G\to \GL_1(\Q_p)$.
Since $G$ has no non-trivial abelian subnormal subgroups, we must have $Z=\one$,
whence $K$ would be discrete and thus $K\not=G$.
Since $\Z_p^\times$ (being pro-finite) and $\Z$ do not contain divisible elements
apart from the neutral element but $\Q_p$ is divisible,
every group homomorphism $\Q_p\to \GL_1(\Q_p)\cong \Q_p^\times \cong \Z\times \Z_p^\times$
is trivial. Since $G$ is generated by copies of $\Q_p$, we deduce that every
group homomorphism $G\to\GL_1(\Q_p)$ is trivial. Thus $K=G$ is non-discrete,
contradiction.}
\item[(iii)]
Our group $G$ does not provide a negative
answer to Problem 1,
as we can choose $K$ as in~(c)
and let $Z$ be the trivial group.
\item[(iv)]
As a consequence of (b),
neither $G$, nor any non-trivial subnormal
subgroup of $G$, is soluble.
\end{itemize}
Following \cite{ELE},
a nearly simple $p$-adic Lie group $H$
is called \emph{extraordinary}
if $H=\ker(\Ad)$ holds and $S'$ is non-discrete
for each open subnormal subgroup $S$ of $H$.
This implies that the radical $R(S)$ is discrete
for each open subnormal subgroup $S$ of $H$
\cite[Remark~3.3]{ELE}.\footnote{The notion
of radical used here is recalled in Section~\ref{preli}.}
It is not known whether extraordinary $p$-adic
Lie groups exist. In Section~\ref{prfprop}, we prove:\\[2.5mm]
{\bf Proposition.}
\emph{If Problem}~1 \emph{has a positive answer,
then extraordinary $p$-adic Lie groups
cannot exist.}\\[3mm]
Being $1$-dimensional,
the Lie group $G$ from the above theorem
is nearly simple. Also, $G=\ker\Ad$
and $R(S)=\one$ (which is discrete),
for each subnormal subgroup
$S$ of $G$.
Thus $G$ shares many properties of
extraordinary groups.
Yet, $G$ is not extraordinary, as $G'$ is discrete
(by (d) in the theorem).
%
%
%
%
\section{Preliminaries}\label{preli}
We write $\N=\{1,2,\ldots\}$ and
$\N_0:=\N\cup\{0\}$.
If $G$ is a group,
then $Z(G)$ denotes its centre.
As usual, we write $G'\sub G$ for the commutator
group of~$G$, generated by the commutators
$ghg^{-1}h^{-1}$ for $g,h\in G$
and define the derived series
of $G$ via $G^{(0)}:=G$, $G^{(n+1)}:=(G^{(n)})'$
for $n\in \N_0$.
We write $\one$ for the trivial group.
As usual, $G$ is called \emph{soluble}
if $G^{(n)}=\one$ for some $n\in \N_0$.
The \emph{radical} $R(G)$ of $G$
is defined as the union of
all soluble normal subgroups of~$G$.
Then $R(G)$ is a normal (and, in fact, characteristic)
subgroup of $G$.\\[2.3mm]
Let $A_1$ and $A_2$
be groups having isomorphic subgroups
$B_1\sub A_1$ and $B_2\sub A_2$,
respectively.
Let $\theta\colon B_1\to B_2$
be an isomorphism.
Then the amalgam $A_1*_\theta A_2$
can be defined (see \cite{Rot}).
We shall always assume
that $B:=B_1=B_2$ and $\theta=\id_B$
is the identity map.
As usual, we abbreviate
\[
A_1*_BA_2:=A_1*_{\id_B}A_2
\]
and call $A_1*_BA_2$ the \emph{free product
of $A_1$ and $A_2$ with amalgamated
subgroup $B$}
(or \emph{amalgamated product}, for short).
We are only interested in central amalgamations,
i.e., we shall always assume that $B\sub Z(A_1)$
and $B\sub Z(A_2)$.
For basic properties
of amalgamated products,
the reader is refered to \cite{Rot}.
In particular, we essentially use
\cite[Theorem 11.66]{Rot}
on the normal form for
elements of $A_1*_BA_2$.
See \cite{Bou} and \cite{Ser}
for basic theory and notation
concerning $p$-adic Lie groups.
We shall write $L(G)$
for the Lie algebra of a $p$-adic
Lie group $G$ and
$\rad(\cg)$ for
the radical of a finite-dimensional
$p$-adic Lie algebra~$\cg$
(the largest soluble ideal).
A $p$-adic Lie group $G$ is called
\emph{linear} if it admits an injective
continuous homomorphism $G\to \GL_d(\Q_p)$
for some $d\in \N$.
If $G$ is linear, then $L(R(G))=\rad(L(G))$
(\cite[Lemma 1.3]{ELE};
cf.\ \cite[Lemma 6.4]{Clu}).
\section{Auxiliary results on amalgamated products}
\begin{la}\label{ga000}
Let $A$ and $C$ be groups and $B$ be a
subgroup of both $Z(A)$ and $Z(C)$.
If $h\in A$ such that $h\not\in B$
and $g\in A*_B C$ such that $g\not\in A$,
then $ghg^{-1}\not\in A$
and $ghg^{-1}h^{-1}\not\in A$.
\end{la}
\begin{proof}
We may assume that the sets $A\setminus B$ and $C\setminus B$
are disjoint.
Let $R\sub A$ be a transversal for $A/B$
and $S\sub C$ be a transversal for $C/B$.
Then $h=rb$ for unique $r\in R$ and $b\in B$.
Moreover,
\[
g=x_1\cdots x_nb'
\]
for some $n\in \N$,
$b'\in B$
and elements $x_1,\ldots, x_n\in R\cup S$
such that $x_i\in R\impl x_{i+1}\in S$
and $x_i\in S\impl x_{i+1}\in R$
for all $i\in \{1,\ldots,n-1\}$.
Since $g\not\in A$, we have $x_i\in S$
for some $i\in \{1, \ldots, n\}$.
We let $j$ be the minimum of all such $i$,
and $k$ be their maximum.
Then $j\in \{1,2\}$, $j\leq k$, $k\in \{n-1,n\}$
and both $x_j,x_k\in S$.

We show that $ghg^{-1}\not\in A$.
Then also $ghg^{-1}h^{-1}\not\in A$ (as $h\in A$).\\[2.3mm]
If $ghg^{-1}\not\in A$ was false,
we would have $gh=ag$ for some $a\in A$.
Note that
\[
gh=x_1\ldots x_n r bb'.
\]

If $k=n$, then
\begin{equation}\label{c1}
gh=x_1\ldots x_n rb_1
\end{equation}
is in normal form, with $b_1:=bb'\in B$.\\[2.3mm]

If $k=n-1$ and $x_nr\in B$, then
\begin{equation}\label{c2}
gh=x_1\ldots x_{n-1}b_2
\end{equation}
is in normal form,
with $b_2:=x_nrbb'\in B$.

If $k=n-1$ and $x_nr\not\in B$,
then $x_nrbb'=r'b_3$
for some $r'\in R$ and $b_3\in B$,
and thus
\begin{equation}\label{c3}
gh=x_1\ldots x_{n-1}r'b_3
\end{equation}
is in normal form.\\[2.3mm]
Suppose that $a\in B$.
We then deduce from
\[
gh=ag =ga
\]
(where we used that $B\sub Z(A*_B C)$)
that $h=a\in B$. But $h\not\in B$, contradiction.
We therefore must have $a\in A\setminus B$
and thus $a=r''b''$
with $r''\in R$ and $b''\in B$.

If $j=1$, this implies that
\begin{equation}\label{c4}
ag=r'' x_1\cdots x_n  b_4
\end{equation}
is in normal form,
with $b_4:= b' b''$.
Then $gh=ag$ is false (contradiction)
because the normal form (\ref{c4})
differs from the ones in (\ref{c1}), (\ref{c2})
and (\ref{c3})
(each starting with a representative $x_1\in S$,
as we assume $j=1$).

If $j=2$ and $r'' x_1\in B$, then
\begin{equation}\label{c5}
ag=x_2\cdots x_n b_5
\end{equation}
is in normal form,
with $b_5:=r'' x_1 b'b''\in B$.
Again we obtain $gh\not=ag$,
as the normal form in (\ref{c5})
differs from the ones in (\ref{c1})
and (\ref{c3}) (which involve $n+1$ and $n$ representatives,
respectively);
it also differs from the normal form in~(\ref{c2})
as the latter ends with a representative in $S$,
while the normal form in~(\ref{c5})
ends with a representative in $R$
in the situation of (\ref{c2}).

If $j=2$ and $r'' x_1\not\in B$,
then
\begin{equation}\label{ennew}
r'' x_1 b'b'' = r_1 b_6
\end{equation}
with $r_1\in R$
and $b_6\in B$, and thus
\begin{equation}\label{c6}
ag=r_1x_2\cdots x_n b_6
\end{equation}
is
in normal form.
Again $gh\not=ag$,
as the normal form in (\ref{c6})
differs from that in
(\ref{c1})
(which involves $n+1$ representatives)
and the one in (\ref{c2})
(which only involves $n-1$ representatives).
It also differs from the normal form in (\ref{c3}).
In fact, if we had $gh=ag$ in the situation of (\ref{c3})
and (\ref{c6}), then
\[
x_1\ldots x_{n-1}r'b_3=
r_1x_2\cdots x_n b_6
\]
would imply that $x_1=r_1$, $r'=x_n$ and $b_3=b_6$.
Thus $r'' x_1 b'b''=r_1 b_6=x_1 b_6$ (using (\ref{ennew})),
whence $b'b'' r'' x_1=b_6x_1$ and thus
$b'b'' r''=b_6$. Hence $r''=(b'b'')^{-1}b_6\in B$,
contrary to $r''\in R\sub A\setminus B$.
So, in all cases, $ghg^{-1}\not\in A$.
\end{proof}
\begin{numba}\label{nwsett}
Throughout the rest of this section,
we consider the following situation:
\begin{itemize}
\item[(a)]
$B_0\supseteq B_1\supseteq \cdots$ is a descending sequence
of groups with $\bigcap_{n\in \N_0}B_n=\one$;
\item[(b)]
$(H_n)_{n\in \N_0}$ is a sequence of groups
such that, for all $n\in \N_0$,
both $Z(H_n)$ and $Z(H_{n+1})$
contain $B_n$ as a subgroup; and
\item[(c)]
$B_n\not=H_n$ and
$B_n\not= H_{n+1}$,
for each $n\in \N_0$.
\end{itemize}
We define $G_0:=H_0$
and $G_n:=G_{n-1}*_{B_{n-1}}H_n$
for $n\in \N$. Then $G_0\sub G_1\sub\cdots$
and we define the group $G$ as the union (direct limit)
\[
G:=\bigcup_{n\in \N_0}G_n=\bigcup_{n\in \N_0}H_0*_{B_0}H_1
*_{B_1}\cdots *_{B_{n-1}}H_n.
\]
\end{numba}
\begin{defn}
A subgroup $H\sub G$ is called \emph{spread out}
if
\[
(\forall n\in \N)\quad H\not\sub G_n.
\]
For example, $G$ is spread out.
\end{defn}
\begin{la}\label{normalspread}
If a subgroup $H\sub G$ is spread
out, then every non-trivial normal subgroup
$N$ of $H$ is spread out.
\end{la}
\begin{proof}
Let $k\in \N_0$. We show that $N$
is not a subset of $G_k$.
Let $e\not=h\in N$
and $m\in \N_0$ such that $h\in G_m$.
After increasing $m$ if necessary, we may assume
that $m\geq k$.
Since $\bigcap_{n\in \N_0}B_n=\one$,
after increasing $m$ further we may assume
that $h\not\in B_m$.
Since $H$ is spread out,
there exists $g\in H$
such that
$g\not \in G_m$.
There exists $\ell>m$
such that $g\in G_\ell$.
We choose $\ell$ minimal.
Then $g\not\in G_{\ell-1}$.
After increasing $m$,
we may assume that $m=\ell-1$.
Thus $g\in G_{m+1}$ but $g\not\in G_m$.
Now $ghg^{-1}\in N$ as $N$ is a normal
subgroup of~$H$.
Moreover, applying Lemma~\ref{ga000}
to the subgroup $G_{m+1}=G_m*_{B_m}H_{m+1}$
of~$G$, we see that $ghg^{-1}\not\in G_m$.
Since $m\geq k$, we deduce that $ghg^{-1}\not\in G_k$.
Thus $N$ is spread out, which completes the proof.
\end{proof}
\begin{rem}\label{sglout}
For later use, note
that also $ghg^{-1}h^{-1}\not\in G_m$
in the preceding proof
(since $ghg^{-1}\not\in G_m$ and $h\in G_m$).
Thus $ghg^{-1}h^{-1}\not\in G_k$
(as $m\geq k$).
\end{rem}
A trivial induction based on
Lemma~\ref{normalspread} shows:
\begin{la}\label{subnspread}
If a subgroup $H$ of $G$ is spread out,
then also every non-trivial subnormal subgroup
$S$ of $H$ is spread out.
In particular, every non-trivial subnormal
subgroup $S$ of $G$ is spread out.\,\Punkt 
\end{la}
\begin{la}\label{spreadna}
If a subgroup $H\sub G$ is spread out,
then $H'\not=\one$ and thus $H$ is not abelian.
\end{la}
\begin{proof}
Let $k\in \N_0$.
Taking $N=H$ in the proof
of Lemma~\ref{normalspread},
we obtain $g,h\in H$
such that $ghg^{-1}h^{-1}\not\in G_k$
(see Remark~\ref{sglout}).
Since $ghg^{-1}h^{-1}\in H'$,
we see that $H'$ is spread out
and thus $H'\not=\one$.
\end{proof}
\begin{la}\label{nicenc}
Every non-trivial
subnormal subgroup $S$ of $G$
is non-abelian
and has trivial radical,
$R(S)=\one$.
\end{la}
\begin{proof}
Being subnormal and non-trivial, $S$ is spread out
(see Lemma~\ref{subnspread})
and hence non-abelian, by Lemma~\ref{spreadna}.
Let $S=S^{(0)}\supseteq S^{(1)}\supseteq S^{(2)}\supseteq \cdots$
be the derives series of $S$.
Then $S^{(n)}$ is a subnormal subgroup of $G$,
for each $n\in \N_0$.
We have $S^{(0)}=S\not=\one$.
If $S^{(n)}\not=\one$,
then $S^{(n)}$ is spread out (by Lemma~\ref{subnspread}),
whence $S^{(n+1)}=(S^{(n)})'\not=\one$
by Lemma~\ref{spreadna}.
Hence $S^{(n)}\not=\one$ for all $n\in \N_0$
and thus $S$ is not soluble.
If $N$ is a non-trivial normal subgroup
of $S$, then also $N$ is subnormal in $G$
and hence $N$ is not soluble,
by the preceding.
Hence $\one$ is the only soluble normal
subgroup of $S$ and hence $R(S)=\one$.
\end{proof}
\section{Proof of the theorem}
We now prove the theorem stated in the Introduction.\\[2.3mm]
First note that $G_0:=\Q_p$ is a $p$-adic Lie group
with $p^{-1}\Z_p$ as a central open subgroup.
If $n\in \N$ and we already know that $G_{n-1}$ is a $p$-adic
Lie group with $p^{n-2}\Z_p$ as a central open subgroup,
then $G_n:=G_{n-1}*_{p^{n-1}\Z_p}\Q_p$
has $p^{n-1}\Z_p$ as a central subgroup.
Now Proposition~18 in \cite[Chapter III, \S1, no.\,9]{Bou}
shows that $G_n$
admits a unique $p$-adic Lie group structure
turning $p^{n-1}\Z_p$ into an open Lie subgroup
of $G_n$. Then $G_{n-1}$ is an open Lie subgroup of $G_n$
and thus
$G:=\bigcup_{n\in \N_0}G_n$
admits a unique Lie group structure
making each $G_n$ an open Lie subgroup.

(a) If $g\in G$, then $g\in G_n$ for some $n\in \N_0$.
Since $p^{n-1}\Z_p \sub Z(G_n)$,
the centre $Z(G_n)$ is open in $G_n$
and thus $\Ad(g)=\id$ (as it does not matter
for the calculation of $\Ad(g)$
if we consider $g$ as an element of $G$,
or as an element of the open subgroup $G_n\sub G$).

(b) is a special case of Lemma~\ref{nicenc}.

(c) To construct a group homomorphism $\phi\colon G\to\Q_p$,
we start with $\phi_0:=\id_{\Q_p}\colon G_0=\Q_p\to \Q_p$, $z\mto z$.
If $n\in \N$ and a homomorphism
$\phi_{n-1}\colon G_{n-1}\to\Q_p$
has already been constructed
with $\phi_{n-1}(z)=z$ for all $z\in p^{n-2}\Z_p$,
then the universal
property of the amalgamated product
provides a unique group homomorphism
\begin{equation}\label{prodc}
\phi_n\colon G_n=G_{n-1}*_{p^{n-1}\Z_p}\Q_p\to \Q_p
\end{equation}
such that
\begin{equation}\label{juxta}
\phi_n|_{G_{n-1}}=\phi_{n-1}
\end{equation}
and $\phi_n|_{\Q_p}=\id_{\Q_p}\colon \Q_p\to\Q_p$
(considered as a map from the second factor
in (\ref{prodc}) to $\Q_p$).
By (\ref{juxta}),
we obtain a well-defined group homomorphism
$\phi\colon G\to\Q_p$ via $\phi(x):=\phi_n(x)$
if $x\in G_n$. Then $\phi|_{G_n}=\phi_n$
for each $n$.
In particular, restricting $\phi$ to
$G_0=\Q_p$, we obtain
$\phi|_{\Q_p}=\phi_0=\id_{\Q_p}$.
Since $G_0$ is open in $G$ and $\id_{\Q_p}$
an analytic map, we deduce that $\phi$
is analytic and $L(\phi)=T_e\id_{\Q_p}=\id_{\Q_p}$.
Thus $\phi$ is \'{e}tale and surjective,
whence $K:=\ker\phi$ is discrete
and $G/K\cong \phi(G)=\Q_p$.
To get $\psi$, compose $\phi$ with the
embedding
$\Q_p\to\GL_2(\Q_p)$, $z\mto${\footnotesize$\left(
\begin{array}{cc}
1& z\\
0 & 1 
\end{array}
\right)$}.

(d) Since $G/K$ is abelian in (c), we have $G'\sub K$.
As $K$ is discrete, discreteness of $G'$ follows.\,\Punkt
\section{Proof of the proposition}\label{prfprop}
We now prove the proposition stated in the Introduction.\\[2.3mm]
To reach a contradiction, suppose that
an extraordinary $p$-adic Lie group $H$ exists.
If Problem~1 has a positive answer, then we can find a normal subgroup
$K\sub H$ such that $H/K$ is linear and
there exists an abelian, normal subgroup $Z\sub K$
with $K/Z$ discrete. Then $Z$ is open in $H$ or discrete.
If $Z$ is open, then $Z$ is an abelian open subnormal
subgroup of $H$; but $Z'$ must be non-discrete as $H$
is extraordinary (contradiction).
If $Z$ is discrete, then also $K$ is discrete.
Because $H/K$ is linear, we have $L(R(H/K))=\rad(L(H/K))$
(see \cite[Lemma 1.3]{ELE};
cf.\ \cite[Lemma 6.4]{Clu})
and thus $L(R(H/K))=L(H/K)$,
since $L(H/K)$ (like $L(H)$) is abelian.
Hence
$R(H/K)$ contains a soluble closed normal
subgroup $S$ of $H/K$
(see \cite[Lemma 1.2]{ELE}).
Let $q\colon H\to H/K$
be the quotient morphism. Then $T:=q^{-1}(S)$ is an
open subnormal subgroup of~$H$
and
$T^{(k)}$ is discrete for large $k$
(because $T^{(k)}\sub K$ for large $k$,
in view of $q(T^{(k)})=S^{(k)}$).
Choose $k$ minimal with $T^{(k)}$ discrete.
Then $N:=T^{(k-1)}$ is an open subnormal subgroup
of $H$ such that $N'$ is discrete,
contradicting the hypothesis that~$H$ is extraordinary.\,\vspace{-2mm}\Punkt
{\small
{\bf Helge  Gl\"{o}ckner}, Universit\"at Paderborn, Institut f\"{u}r Mathematik,\\
Warburger Str.\ 100, 33098 Paderborn, Germany\\[1mm]
e-mail: {\tt  glockner\at{}math.upb.de}}\vfill
\end{document}